\documentclass[11pt]{article}
\usepackage{amssymb,comment}
\usepackage{amsmath}
\usepackage{graphicx}
\usepackage{color}
\usepackage{cite}
\usepackage[thinlines]{easytable}
\newcounter{foo}

\newfont{\blb}{msbm10 scaled\magstep1}
\newfont{\comp}{cmr12 scaled\magstep1}
\newfont{\compb}{cmr10 scaled\magstep2}
\newfont{\sbb}{cmssbx10 scaled\magstep3}
\newfont{\sbbb}{cmssbx10 scaled\magstep5}
\newfont{\sbs}{cmssbx10 scaled\magstep1}
\newtheorem{theorem}{Theorem}
\newtheorem{lemma}{Lemma}

\newtheorem{proposition}{Proposition}
\newenvironment{proof}{{\bf Proof.}}{\hfill{ }\vrule height10pt width5pt depth1pt}
\newtheorem{conjecture}[foo]{Conjecture}

\title{Ramsey numbers and the Zarankiewicz problem}

\author{David Conlon\thanks{Department of Mathematics, California Institute of Technology, Pasadena, CA 91125, USA. Email: {\tt dconlon@caltech.edu}. Research supported by NSF grant DMS-2054452.}
\and
Sam Mattheus\thanks{Department of Mathematics, Vrije Universiteit Brussel, Brussels, Belgium.  E-mail: {\tt sam.mattheus@vub.be}. Research as a Visiting Scholar at UCSD supported by a Fulbright Visiting Scholar Fellowship and a
    Fellowship of the Belgian American Foundation.}
\and
Dhruv Mubayi\thanks{Department of Mathematics, Statistics and Computer Science, University of Illinois Chicago, Chicago, IL 60607, USA. E-mail: {\tt mubayi@uic.edu}. Research supported by NSF grants
DMS-1763317, DMS-1952767, DMS-2153576, a Humboldt Research Award and a Simons Fellowship.}
\and
Jacques Verstra\"ete\thanks{Department of Mathematics, University of California, San Diego, La Jolla, CA 92093, USA. E-mail: {\tt jacques@ucsd.edu}. Research supported by NSF grant DMS-1952786.}
}

\date{}

\begin{document}

\maketitle

\vspace{-0.3in}

\begin{abstract}
Building on recent work of Mattheus and Verstra\"ete, we establish a general connection between Ramsey numbers of the form $r(F,t)$ for $F$ a fixed graph 
and a variant of the Zarankiewicz problem 
asking for the maximum number  of 1s in an $m$ by $n$ $0/1$-matrix that does not have any matrix from a fixed finite family $\mathcal{L}(F)$ derived from $F$ as a submatrix.
As an application, we give new lower bounds for the Ramsey numbers $r(C_5,t)$ and $r(C_7,t)$, namely, 
$r(C_5,t) = \tilde\Omega(t^{\frac{10}{7}})$ and $r(C_7,t) = \tilde\Omega(t^{\frac{5}{4}})$. 
We also show how the truth of a plausible conjecture about Zarankiewicz numbers would allow an approximate determination of $r(C_{2\ell+1}, t)$ for any fixed integer $\ell \geq 2$.
\end{abstract}

\section{Introduction}

For a graph $F$ and an integer $t \geq 2$, let $r(F,t)$ denote the minimum $n$ such that every $n$-vertex $F$-free graph contains an independent set of order $t$. In this note, we will be concerned with the case where $F$ is fixed and $t \rightarrow \infty$. Basic estimates on such {\em $F$-complete Ramsey numbers} go back to foundational work of Erd\H{o}s and Szekeres~\cite{ES} in the 1930s, though the asymptotics are understood in very few cases. Very recently, answering a longstanding question in the area, Mattheus and Verstra\"ete~\cite{MattheusV} showed that $r(K_4,t) = \tilde\Omega(t^3)$, which is tight up to the logarithmic factors hidden in the big O notation.  
Here, building on their approach, we establish a general connection between the Ramsey numbers $r(F, t)$ and a variant of the Zarankiewicz problem~\cite{Z} and use it to improve the lower bounds on certain cycle-complete Ramsey numbers.

\medskip

{\bf Main Result.}  To describe the connection between Ramsey numbers and the Zarankiewicz problem, suppose $J$ is a bipartite graph with parts $U$ and $V$. The {\em incidence matrix} $A = A(J)$ of $J$ is the matrix with rows indexed by $U$ and columns indexed by $V$ where $A_{uv} = 1$ if $uv \in E(J)$ and $A_{uv} = 0$ otherwise. The {\em Zarankiewicz number} $z(m,n,A)$ is then the maximum number of 1s in an $m$ by $n$ 0/1-matrix that does not have $A$ as a submatrix. In graph-theoretic terms, this is the maximum number of edges in an $m$ by $n$ bipartite graph which does not contain a copy of $J$ with $U$ in the part of order $m$ and $V$ in the part of order $n$. 

Given a graph $F$, let $\mathcal{E}(F)$ be the set of sequences of edge-disjoint bipartite subgraphs $H_1,H_2,\dots,H_k$ of $F$, each containing at least one edge, such that the $E(H_i)$ partition $E(F)$ and $|V(H_i) \cap V(H_j)| \leq 1$ for all $1 \leq i < j \leq k$. Given $H = (H_1,H_2,\dots,H_k) \in \mathcal{E}(F)$,
let $J = J(H)$ be the bipartite graph with parts $[k]$ and $V(F)$ where $i \in [k]$ is adjacent in $J$ to every vertex in $V(H_i) \subseteq V(F)$. For instance, if $k =|E(F)|$, then $H$ is an ordering of $E(F)$ and $J(H)$ is the so-called {\em 1-subdivision} of $F$.
We then let $\mathcal{L}(F) = \{J(H) : H \in \mathcal{E}(F)\} \cup \{C_4\}$. Abusing notation, we will also use $\mathcal{L}(F)$ to denote the collection of incidence matrices of the graphs in $\mathcal{L}(F)$.  

Our main result connects the problem of estimating the Ramsey number $r(F, t)$ to that of estimating the Zarankiewicz number $z(m,n,\mathcal{L}(F))$, the maximum number of $1$s in an $m$ by $n$ $0/1$-matrix which does not contain 
any member of $\mathcal{L}(F)$ as a submatrix. To state the result, it will be convenient to define an {\em $(m,n,a,b)$-graph} to be an $m$ by $n$ bipartite graph with $n \geq m$ such that all vertices in the part of size $n$ have degree $a$ and all vertices in the part of size $m$ have degree $b \geq a$. Here and throughout, unless otherwise indicated, $\log$ will be taken to the base $2$.

\begin{theorem}\label{main}
Let $F$ be a graph and let $a,b,m,n$ be positive integers with $a \geq 2^{12}(\log n)^3$ such that there exists an $\mathcal{L}(F)$-free $(m,n,a,b)$-graph. If  $t = 2^8n(\log n)^2/ab$, then 
\begin{equation*}
r(F,t) = \Omega\left(\frac{bt}{\log n}\right).
\end{equation*}
\end{theorem}

The logarithmic factors in Theorem~\ref{main} can be improved for many choices of the parameters, but, since the improved bound is also unlikely to give the correct exponent on these factors, we opted to give a simpler proof with slightly worse bounds. Though perhaps not a surprise, since our methods are based on those of~\cite{MattheusV}, we also note that Theorem~\ref{main} gives the bound $r(K_4,t) = \tilde\Omega(t^3)$ if we make use of the $\mathcal{L}(K_4)$-free graph with parameters $(m,n,a,b) = (q^3+1, q^4-q^3+q^2, q+1, q^2)$ as in~\cite{MattheusV}. Moreover, if some natural conjectures on $z(m,n,\mathcal{L}({K_s}))$ are true, then the theorem also implies that $r(K_s,t) = \tilde\Omega(t^{s-1})$ for all $s \ge 4$.
We will return to these points in the concluding remarks.

\medskip

{\bf Cycle-complete Ramsey numbers.} Our main new applications of Theorem~\ref{main} are to the estimation of the cycle-complete Ramsey numbers $r(C_k, t)$. The study of such quantities goes back to works of Bondy and Erd\H{o}s~\cite{BE73} and  Erd\H{o}s, Faudree, Rousseau and Schelp~\cite{EFRS} in the 1970s. In general, the best known upper bounds~\cite{CLRZ, Sudakov} for $r(C_k, t)$ for fixed $k \geq 2$ are that $r(C_{2\ell},t) = \tilde{O}(t^{\ell/(\ell - 1)})$ and $r(C_{2\ell + 1},t) = \tilde{O}(t^{(\ell+1)/\ell})$, where the logarithmic factors hidden inside the big O notation have exponents  
depending only on $\ell$. Up until recently, the lower bounds have always come from applications of the probabilistic method, with the best result, proved by Bohman and Keevash~\cite{BK10} using the $C_k$-free process, saying that, for fixed $k \ge 3$,
\[ r(C_k,t) = \Omega\left(\frac{t^{\frac{k - 1}{k - 2}}}{\log t}\right).\]
That this is optimal when $k = 3$ is an older result, with the lower bound due to Kim~\cite{K} and the upper bound to Ajtai, Koml\'{o}s and Szemer\'{e}di~\cite{AKS} (though see also~\cite{BK21, FGM, Sh} for considerable further work on determining the constant factors).

In recent work, Mubayi and Verstra\"ete~\cite{MV} used pseudorandom graphs to improve the bound above by a power of log 
for all odd $k \geq 5$ and some even $k$. For instance, for odd $k \ge 3$, they showed that
\[ r(C_k,t) = \Omega\left(\frac{t^{\frac{k - 1}{k - 2}}}{(\log t)^{\frac{2}{k-2}}}\right).\]

Moreover, using a different approach more closely allied to that of this paper, they were able to give polynomial improvements on $r(C_5,t)$ and $r(C_7,t)$, namely, $r(C_5,t) = \Omega(t^{\frac{11}{8}})$ and $r(C_7,t) = \Omega(t^{\frac{11}{9}})$. 
Using Theorem~\ref{main}, we improve these bounds even further. 

\begin{theorem} \label{c5c7}
As $t \rightarrow \infty$,
\[r(C_5,t) = \Omega\left(\frac{t^{\frac{10}{7}}}{(\log t)^{\frac{13}{7}}}\right) \qquad \hbox{ and } \qquad 
r(C_7,t) = \Omega\left(\frac{t^{\frac{5}{4}}}{(\log t)^{\frac{3}{2}}}\right).
\]
\end{theorem}

To derive Theorem~\ref{c5c7} from Theorem~\ref{main}, we first note that an $m$ by $n$ bipartite graph of girth at least $4\ell + 4$
has $O((mn)^{\beta} + m + n)$ edges, where $\beta = (\ell + 1)/(2\ell + 1)$. This result was proved by Hoory\cite{Hoory}, though we refer the reader to the survey~\cite{V} for more on related earlier results. The following straightforward corollary of Theorem~\ref{main} now says that if one has a construction of an $m$ by $n$ bipartite graph of girth at least $4\ell + 4$ which comes close to meeting Hoory's upper bound, then one also has a good lower bound for $r(C_{2\ell+1},t)$.

\begin{theorem}\label{cycles}
Let $\ell \in \mathbb{N}$, $\alpha_0 \geq 1$ 
and $0 < \beta \leq (\ell + 1)/(2\ell + 1)$ be fixed. If there is an infinite sequence of $(m,n,a,b)$-graphs of girth larger than $4\ell + 2$ with $\Omega((mn)^{\beta})$ edges and $1 \leq \alpha = \log_a b \leq \alpha_0$, then
\[ r(C_{2\ell+1},t) = \Omega\left(\frac{t^{\frac{\alpha \beta - 3\beta + 2}{\alpha - 3\beta - \alpha \beta + 2}}}{(\log t)^{\frac{3\alpha \beta - \alpha - 3\beta + 2}{\alpha - 3\beta - \alpha \beta + 2}}}\right)
\]
for an infinite sequence of values of $t$.
\end{theorem} 

To apply this theorem, we use the bipartite incidence graphs of sequences of generalized polygons~\cite{vM}. For instance, if we set $\ell = 1$ and $\beta = 2/3$, the value of $\alpha$ is unimportant and we can use any of the known sequences of generalized quadrangles to return a bound of the form $r(C_3, t) = \Omega(t^2/\log^3 t)$ along an infinite sequence of values of $t$. Theorem~\ref{c5c7} is a direct consequence of applying Theorem~\ref{cycles} to the bipartite incidence graphs of the generalized hexagons of order $(q,q^3)$ for the $C_5$-case and the generalized octagons of order $(q,q^2)$ for the $C_7$-case. Such generalized polygons are known to exist for appropriate prime powers $q$ and their bipartite incidence graphs are $\mathcal{L}(C_5)$-free and $\mathcal{L}(C_7)$-free, respectively. In the $C_5$-case, we may therefore apply Theorem \ref{cycles} with $\beta = 3/5$ and, since the bipartite incidence graph of a generalized hexagon of order $(q,q^3)$ has $a = q + 1$ and $b = q^3 + 1$, $\alpha = \log(q^3 + 1)/\log(q + 1)$. Similarly, for $C_7$, we may take $\beta = 4/7$ and $\alpha = \log(q^2 + 1)/\log(q + 1)$. This gives Theorem~\ref{c5c7} along certain infinite sequences of values of $t$, but these sequences are sufficiently dense that the same result carries over to all $t$. 

Though the celebrated Feit--Higman theorem~\cite{FH} precludes the existence of thick generalized polygons beyond generalized octagons and, even when they do exist, the possible orders of generalized polygons are heavily constrained, it may still be the case that sequences of graphs of the type required by Theorem~\ref{cycles} exist. We make a (tentative) conjecture to that effect.

\begin{conjecture} \label{cycleconj}
For every $\ell \in \mathbb{N}$ and $\alpha \geq 1$, there exists an infinite sequence of $(m,n,a,b)$-graphs of girth larger than $4\ell + 2$ with $\Omega((mn)^{(\ell + 1)/(2 \ell +1)})$ edges and $\log_a b \rightarrow \alpha$.
\end{conjecture}

If such a sequence does exist for a given $\ell$ and $\alpha$, we can use Theorem~\ref{cycles} to conclude that 
\[r(C_{2\ell+1}, t) \geq t^{\frac{\alpha (\ell + 1) + \ell - 1}{\alpha \ell + \ell - 1} - o(1)}\]
along some infinite sequence of values of $t$. Since $(\alpha (\ell + 1) + \ell - 1)/(\alpha \ell + \ell - 1) \rightarrow (\ell+1)/\ell$ as $\alpha \rightarrow \infty$, taking $\alpha$ sufficiently large in terms of $\epsilon$ yields the following result.

\begin{theorem}
If Conjecture~\ref{cycleconj} holds for a fixed $\ell \in \mathbb{N}$ and a sequence of $\alpha$ tending to infinity, then, for any $\epsilon > 0$, 
\[r(C_{2\ell+1}, t) \geq t^{(\ell+1)/\ell - \epsilon}\]
along some infinite sequence of values of $t$.
\end{theorem}

That is, if Conjecture~\ref{cycleconj} holds for a particular $\ell$ and a sequence of $\alpha$ tending to infinity, we can prove lower bounds which come arbitrarily close to matching the polynomial factor in the upper bound $r(C_{2\ell+1}, t) = \tilde{O}(t^{(\ell+1)/\ell})$. To us, this strongly suggests that $(\ell + 1)/\ell$ is the correct exponent.

\section{Proof of Theorem \ref{main}}

In this section, we prove Theorem~\ref{main}. When necessary, we assume that relevant quantities are integers without introducing rounding, as this does not change the statements of the results. For a multiset $\mathcal{S}$ 
of edge-disjoint cliques in a graph, let $v(\mathcal{S})$ denote the number of pairs $(x,T)$ with $T \in \mathcal{S}$ and $x \in V(T)$ and let $E(\mathcal{S}) = \bigcup_{T \in \mathcal{S}} E(T)$ and $e(\mathcal{S}) = |E(\mathcal{S})|$. The quantity $v(\mathcal{S})$ can also be interpreted as the number of edges in the bipartite graph with parts $\bigcup_{T \in \mathcal{S}} V(T)$  and $\mathcal{S}$ where $x \in \bigcup_{T \in \mathcal{S}} V(T)$ is adjacent to $T \in \mathcal{S}$ if and only if $x \in T$. 
The following simple lemma gives a lower bound on $e(\mathcal{S})$ in terms of $v(\mathcal{S})$ under appropriate conditions. 

\begin{lemma}\label{cliques}
If $\mathcal{S}$ is a multiset of edge-disjoint cliques in a graph with $v(\mathcal{S}) \geq 2|\mathcal{S}|$, then
\[ e(\mathcal{S}) \geq \frac{v(\mathcal{S})^2}{4|\mathcal{S}|}.\]
\end{lemma}

\begin{proof} 
Since the cliques in $\mathcal{S}$ are edge-disjoint, 
\[ e(\mathcal{S}) = \sum_{T \in \mathcal{S}} {|V(T)| \choose 2}.\]
Thus, by Jensen's Inequality, 
\[ e(\mathcal{S}) \geq |\mathcal{S}| \cdot {\frac{1}{|\mathcal{S}|} \sum_{T \in \mathcal{S}} |V(T)| \choose 2} = |\mathcal{S}| \cdot {\frac{v(\mathcal{S})}{|\mathcal{S}|} \choose 2} = \frac{v(\mathcal{S})(v(\mathcal{S}) - |\mathcal{S}|)}{2|\mathcal{S}|} \geq \frac{v(\mathcal{S})^2}{4|\mathcal{S}|},\]
where we used that $v(\mathcal{S}) \geq 2|\mathcal{S}|$.  
\end{proof}

\bigskip

For positive integers $m \leq n$ and $a \leq b$, an {\em $(m,n,a,b)$-clique graph} is an $n$-vertex graph consisting of an edge-disjoint union of $m$ cliques of order $b$ such that each vertex is in exactly $a$ of the cliques. 
This is equivalent to the restriction of the {\em distance-two graph} of a $C_4$-free $(m,n,a,b)$-graph to the part of size $n$. 
The main input needed for the proof of Theorem~\ref{main} is the following result, which says that if we take a subgraph of an $(m,n,a,b)$-clique graph by taking a random spanning bipartite subgraph of each clique, then the resulting graph is well-distributed with some positive probability.

\begin{lemma}\label{mainlem}
Let $H$ be an $(m,n,a,b)$-clique graph with $a \geq 2^{12}(\log n)^3$ and let $H^*$ be obtained by independently taking a random spanning bipartite subgraph of each clique of $H$. Then, with positive probability, each $X \subseteq V(H^*)$ with $|X| \geq R := 2^{10}(m\log n)/a$ has
\begin{equation*}
 e(H^*[X]) \geq S := \frac{a^2}{2^{8}m}|X|^2.
 \end{equation*}
\end{lemma}

\begin{proof}
By a standard sampling argument (as, for instance, in~\cite[Section 3.2]{MattheusV}), it suffices to prove the result for all $X$ of order exactly $R$. Let $\mathcal{T}$ denote the collection of $m$ cliques of order $b$ which comprise $H$. For a given set $X$ of order $R$, we consider the multiset of cliques $$\mathcal{T}_X = \{H[V(T) \cap X] : |V(T) \cap X| \ge 1, T \in \mathcal{T}\}$$
and, for $1\leq i \leq \log n$, we set
\begin{eqnarray*}
\mathcal{T}_i = \{T \in \mathcal{T}_X : 2^{i-1} a < |V(T)| \leq 2^{i}a\} \quad \mbox{ and } \quad \mathcal{T}_0 = \{T \in \mathcal{T}_X : 1 \leq |V(T)| \leq a\}.
\end{eqnarray*}

Then $v(\mathcal{T}_X) = v(\mathcal{T}_0) + v(\mathcal{T}_1) + \cdots$ and $v(\mathcal{T}_X) = 
a|X|$. We now split into two cases.

\medskip

{\bf Case 1:} $v(\mathcal{T}_0) \geq a|X|/2$. Since $a|X|/2 \geq 2m \geq 2 |\mathcal{T}_0|$, we can apply Lemma \ref{cliques} with $\mathcal{S} = \mathcal{T}_0$ to obtain 
\begin{equation*}
e(\mathcal{T}_0) \geq \frac{a^2|X|^2}{16m} \geq 4S.
\end{equation*}
Let $Z_0$ be the number of edges of $H^*$ in cliques from $\mathcal{T}_0$, noting that $\mathbb{E}(Z_0) \geq \frac 12 e(\mathcal{T}_0)$. Then, as in~\cite[Section 3.2]{MattheusV}, 
$Z_0$ may be viewed as the final term of a martingale. Explicitly, for each clique $T \in \mathcal{T}_0$ and each $v \in V(T)$, we define a random variable $Z_{v, T}$ taking values $0$ or $1$, each with probability $1/2$, to indicate which side of the random cut of $T$ the vertex $v$ is in. Then, given an ordering $Y_1, \dots, Y_k$ with $k = v(\mathcal{T}_0)$ of all such random variables, we let $\mathcal{F}_j$ be the $\sigma$-algebra generated by $Y_1, \dots, Y_j$ and consider the (Doob) martingale $Z_{0,j} = \mathbb{E}(Z_0 | \mathcal{F}_j)$. Since $Z_0$ is a function of the $Z_{v,T}$, we have that $Z_0 = Z_{0,k}$. Hence, as fixing the placement of a given vertex in a cut of a given clique $T$ can change the expectation by at most $|V(T)| - 1$, we
may apply the Hoeffding--Azuma inequality to conclude that 
\[
\mathbb P(Z_0 < S) \le \mathbb{P}(Z_0 - \mathbb{E}(Z_0) < - \frac 14 e(\mathcal{T}_0)) \leq \exp\left(-\frac{e(\mathcal{T}_0)^2}{32 \sum_{T \in \mathcal{T}_0} \sum_{v \in T} (|V(T)| - 1)^2}\right).\] 
But 
\[\sum_{T \in \mathcal{T}_0} \sum_{v \in T} (|V(T)| - 1)^2 = \sum_{T \in \mathcal{T}_0} |V(T)| (|V(T)| - 1)^2 \leq 2a \sum_{T \in \mathcal{T}_0} {|V(T)| \choose 2} = 2a \cdot e(\mathcal{T}_0).\]
Therefore,
\[\mathbb P(Z_0 < S) \leq \exp\left(-\frac{e(\mathcal{T}_0)}{64 a}\right) \leq n^{-R} < \frac{1}{2{n \choose R}}. 
\]

\medskip

{\bf Case 2:} $v(\mathcal{T}_i) \geq a|X|/2\log n$ for some $1 \leq i \leq \log n$. In this case, we note that
\[ 2^{i - 1}a|\mathcal{T}_i| \leq v(\mathcal{T}_i) \leq a|X|\] 
and so $|\mathcal{T}_i| \leq 2^{11}m\log n/2^i a$. Since $a|X|/2\log n \geq 2 m \geq 2 |\mathcal{T}_i|$, we can apply Lemma \ref{cliques} with $\mathcal{S} = \mathcal{T}_i$ to obtain
\begin{equation*} 
e(\mathcal{T}_i) \geq \frac{a^2 |X|^2}{16 |\mathcal{T}_i| (\log n)^2} = \frac{16 m}{|\mathcal{T}_i| (\log n)^2} \cdot S >\frac{2^i a}{2^{8}(\log n)^3} \cdot S.
\end{equation*}
Let $Z_i$ be the number of edges of $H^*$ in cliques from $\mathcal{T}_i$.
Applying Hoeffding--Azuma as before, we get that
\begin{equation*} \mathbb P\left(Z_i < \frac{2^i a}{2^{10}(\log n)^3} \cdot S\right) \leq \mathbb{P}(Z_i - \mathbb{E}(Z_i) < - \frac 14 e(\mathcal{T}_i)) \leq \exp\left(-\frac{e(\mathcal{T}_i)}{2^{i+6} a}\right) \leq n^{-R} < \frac{1}{2{n \choose R}}. 
\end{equation*}

Combining the two cases and using that $a \geq 2^{12}(\log n)^3$, we see that, with positive probability, $e(H^*[X]) \geq S$ for all $X$ of order $R$.
\end{proof}

\medskip
\bigskip

The second ingredient needed for the proof of Theorem \ref{main} is the following result of Kohayakawa, Lee, R\"{o}dl and Samotij~\cite{KLRS}, which is itself proved using an early version of the hypergraph container method due to Kleitman and Winston \cite{KW} (see also the survey \cite{BMS18} for more on this method and its applications). In the statement, we write $e(X)$ for the number of edges induced by a vertex subset $X$ of a graph.

\begin{proposition}\label{count}
	Let $G$ be a graph on $n$ vertices, let $r,R \in \mathbb{N}$ and let $\delta \in [0,1]$ with
	$e^{-\delta r}n \leq R$ and, for every subset $X \subseteq V(G)$ of at least $R$ vertices,	$2e(X) \geq \delta|X|^2$. 
	Then the number of independent sets of size $t \geq r$ in $G$ is at most
	\begin{equation*}\label{eq:finalcount}
 {n \choose r} {R \choose t - r}.
 \end{equation*}
\end{proposition}

We are now ready to prove our main theorem.

\medskip

{\bf Proof of Theorem \ref{main}.}
Let $G$ be an $\mathcal{L}(F)$-free $(m,n,a,b)$-graph and let $H$ be the $(m,n,a,b)$-clique graph of $G$, 
which is the restriction of the distance-two graph of $G$ to the part of $G$ of size $n$. Let $H^*$ be the subgraph of $H$ guaranteed by Lemma \ref{mainlem}, noting that $H^*$ is $F$-free since $G$ is $\mathcal{L}(F)$-free -- indeed, this is the motivation behind the definition of $\mathcal{L}(F)$. If we let $R = 2^{10}m(\log n)/a$ and $r = t/\log n$, then, by Lemma \ref{mainlem} and the fact that $t = 2^8 n (\log n)^2/ab$, the conditions of Proposition~\ref{count} are satisfied with $\delta = a^2/2^8m$. 
Therefore, by Proposition \ref{count}, the number of independent
sets of order $t$ in $H^*$ is at most
\[ {n \choose r} {R \choose t - r} \leq n^r \left(\frac{eR}{t}\right)^t \leq \left(\frac{2^{10} e^2 m \log n}{at}\right)^t.\]
We now randomly sample the vertices of $H^*$ independently with probability $at/(2^{13} m \log n)$ and then delete one vertex from each independent set of order $t$ that remains, noting that the expected number of remaining vertices is at least $atn/2^{13}m\log n - (e^2/8)^t = \Omega(atn/m\log n)$. If $V$ is such a set of vertices, then $H^*[V]$ contains no independent set of order $t$. Since $H^*[V]$ is also  
$F$-free, we conclude that
\[ r(F,t) = \Omega\left(\frac{atn}{m\log n}\right).\]
As $an = bm$, this is the same as the required bound. {\hfill{ }\vrule height10pt width5pt depth1pt}

\section{Proof of Theorem~\ref{cycles}}


We now give the short proof of Theorem~\ref{cycles}, our main result about cycle-complete Ramsey numbers.

\medskip

{\bf Proof of Theorem~\ref{cycles}.}
By assumption, we have an infinite sequence of $(m,n,a,b)$-graphs of girth larger than $4\ell+2$ with $c(mn)^\beta$ edges for some $c > 0$ and $a \leq b = a^\alpha$. Observe that if a bipartite graph has girth larger than $4\ell+2$, it is ${\cal L}(C_{2\ell+1})$-free. Since the number of edges is $na = mb$, we can write $a = cm^\beta n^{\beta-1}$ and $b = cm^{\beta-1}n^\beta$, which, in combination with $b = a^\alpha$, yields $m = c_1n^{(\alpha+\beta-\alpha\beta)/(\alpha\beta-\beta+1)}$ for some positive constant $c_1$. Using these expressions for $a$, $b$ and $m$, one can verify that, for some positive constant $c_2$, we have
\begin{align*} 
	t = 2^8\frac{n(\log n)^2}{ab} = c_2n^{\frac{\alpha-3\beta-\alpha\beta+2}{\alpha\beta-\beta+1}}(\log n)^2. \label{tvalue}
\end{align*}
But then, for some positive constants $c_3$ and $c_4$, we have
\begin{align*}
	\frac{bt}{\log n} &= \frac{(cm^{\beta-1}n^\beta)(c_2n^{\frac{\alpha-3\beta-\alpha\beta+2}{\alpha\beta-\beta+1}}(\log n)^2)}{\log n} \\
	&=c_3 n^{\frac{\alpha+\beta-\alpha\beta}{\alpha\beta-\beta+1}(\beta-1)}n^\beta n^{\frac{\alpha-3\beta-\alpha\beta+2}{\alpha\beta-\beta+1}}\log n \\
	&=c_3 n^{\frac{\alpha\beta-3\beta+2}{\alpha\beta-\beta+1}} \log n \\
	&\geq c_4 t^{\frac{\alpha\beta-3\beta+2}{\alpha-3\beta-\alpha\beta+2}} (\log t)^{\frac{\alpha+3\beta-3\alpha\beta-2}{\alpha-3\beta-\alpha\beta+2}},	
\end{align*}	
where in the second equality we used the expression for $m$ and in the last we used that $\log t = \Theta(\log n)$. Applying Theorem~\ref{main} therefore yields the required infinite sequence of values of $t$. {\hfill{ }\vrule height10pt width5pt depth1pt}

\section{Concluding remarks}

$\bullet$ Unfortunately, the methods of this note do not help with estimating $r(F,t)$ when $F$ is bipartite, since the graph $H^*$ constructed in the proof of Theorem~\ref{main} contains very large complete bipartite graphs. In particular, the order of magnitude of $r(C_4,t)$ is a wide open problem, with the best bounds~\cite{BK10, CLRZ} being $c_1 t^{3/2}/\log t \leq r(C_4,t) \leq c_2 t^2/\log^2 t$ for some constants $c_1,c_2 > 0$.

\bigskip

$\bullet$  
Though Theorem~\ref{main} is stated and proved under the assumption that there is an $\mathcal{L}(F)$-free bipartite graph which is regular of degree $a$ on one side $V$ of order $n$ and regular of degree $b$ on the other side $U$ of order $m$, we do not really need such a rigid assumption. For one thing, the proof does not depend at all on $U$ being regular and we can simply assume that $b$ is the average degree of $U$. Moreover, the proof is easily adapted to the case where $V$ has all degrees in the interval $(a/2, a]$, say. This gives us significant additional leeway in finding constructions to which we can apply Theorem~\ref{main}.

\bigskip

$\bullet$ 
We now say a little about the off-diagonal Ramsey numbers $r(s, t) := r(K_s, t)$, where we think of $s$ as fixed and $t$ tending to infinity. By adapting arguments of Janzer~\cite{Janzer}, one can show that, for any $s \geq 3$, the maximum number of edges in an $\mathcal{L}(K_s)$-free $m$ by $n$ bipartite graph is
\[ O(m^{\frac{s-1}{2s-3}}n^{\frac{2s-4}{2s-3}} + m + n).\]
If this is tight, that is, if there exists an infinite sequence of $\mathcal{L}(K_s)$-free $(m, n, a, b)$-graphs with $\Omega(m^{\frac{s-1}{2s-3}}n^{\frac{2s-4}{2s-3}})$ edges and $1 \leq \log_a b \leq \alpha_0$ for some $\alpha_0 \ge 1$, then an application of Theorem~\ref{main} implies that $r(s,t) = \tilde{\Omega}(t^{s - 1})$ along an infinite sequence of values of $t$. Surprisingly, unlike the case of odd cycles, the particular choice of $\alpha = \log_a b$ is largely irrelevant to this deduction. Indeed, this is why, in~\cite{MattheusV}, it is possible to start with an $\mathcal{L}(K_4)$-free $m$ by $n$ bipartite graph with $m = \Theta(n^{3/4})$ and still derive an almost tight bound for $r(4,t)$. 

\bigskip

$\bullet$ Lemma \ref{mainlem} shows that the graph $H^*$ is in a sense close to optimally pseudorandom:
the density of edges in $H^*$ is $p = \Theta(ab/n) = \Theta(a^2/m)$ and we have shown that, for some $c > 0$, every set $X$ of order at least $c(m\log n)/a$
contains $\Omega(p|X|^2)$ edges. With more technical effort, as in~\cite{MattheusV} for the case $F = K_4$, one can often improve this bound to an optimal result saying that every set $X$ of order at least $cm/a$ contains $\Omega(p|X|^2)$ edges. This in turn would allow us to save some of the missing logarithmic factors in Theorems~\ref{main}, \ref{c5c7} and~\ref{cycles}. However, as observed in~\cite{MattheusV}, the more significant barrier to obtaining the correct power of $\log t$ in these results lies within Proposition~\ref{count}, the container-type result of Kohayakawa, Lee, R\"{o}dl and Samotij~\cite{KLRS}. We refer the interested reader to~\cite{MattheusV} for more on this issue.

\end{document}